\documentclass[11pt]{article}
\usepackage{amsfonts,amsmath,amssymb,amsthm}
\usepackage{graphicx}
\usepackage{xcolor}
\usepackage[top=1.25in, bottom=1.5in, left=1.5in, right=1.5in]{geometry}

\usepackage{tikz}
\usetikzlibrary{shapes}

\tikzstyle{mybox} = [draw=black, fill=white, thick,
    rectangle]
\newcommand{\fancybox}[1]{%
\begin{tikzpicture}
\node [mybox] (box){%
    \begin{minipage}[t]{0.56\textwidth}
    #1
    \end{minipage}
};
\end{tikzpicture}%
}

\newcommand{\zed}{\ensuremath{\mathbb{Z}}}

\newcommand{\B}{\ensuremath{\mathcal{B}}}
\newcommand{\pc}{\rho}

\newcommand{\STS}{\ensuremath{\mathcal{S}}}
\newcommand{\Par}{\ensuremath{\mathcal{P}}}
\newcommand{\are}{\ensuremath{\mathcal{R}}}

\definecolor{grey}{rgb}{0.71,0.4,0.11}

\newtheorem{Theorem}{Theorem}[section]

\newtheorem{Lemma}[Theorem]{Lemma}

\newtheorem{Corollary}[Theorem]{Corollary}
\newtheorem{Remark}[Theorem]{Remark}

\title {\bf  On maximum parallel classes in packings}
 \author{Douglas R. Stinson\thanks{D.R.\ Stinson's research was supported by  
Natural Sciences and Engineering Research Council of Canada discovery grant RGPIN-03882.
}\\
 David R.\ Cheriton School of Computer Science
 \\
  University of Waterloo\\
 Waterloo, Ontario, N2L 3G1, 
  Canada\\ \texttt{dstinson@uwaterloo.ca}\\
 \and 
 Ruizhong Wei\thanks{R.\ Wei's research was supported by  
Natural Sciences and Engineering Research Council of Canada discovery grant RGPIN-05610.
}\\
 Department of Computer Science\\
 Lakehead Univeristy\\
 Orillia, Ontario, L3V 0B9,  Canada\\
 \texttt{rwei@lakeheadu.ca}
 }

\date{\today}

\begin{document}
\maketitle

\begin{abstract}
  The integer $\beta ( \pc, v, k)$ is defined to be the maximum number of blocks in any $(v, k)$-packing in which
  the maximum partial parallel class (or PPC) has size $\pc$. This problem was introduced and studied by Stinson in \cite{stinson} for the case $k=3$. Here, we mainly consider the case  $k = 4$ and we obtain some
  upper bounds and lower bounds on  $\beta (\pc, v, 4)$.
   We also provide some explicit constructions of $(v,4)$-packings
 having a maximum PPC of a given size $\pc$. For small values of $\pc$, the number of blocks of the constructed packings are very close 
 to the upper bounds on  $\beta (\pc, v, 4)$. Some of our methods are extended to the cases  $k > 4$.
\end{abstract}

\section{Introduction}

For positive integers $v$ and $k$ with $2 \leq k \leq v-1$, a  $(v, k)$-{\em packing} is a pair $\STS = (X, \B)$, where $X$ is a set of $v$ points
and $\B$ is a set of $k$-subsets of $X$, called {\em blocks}, 
such that every pair of points occurs in at most one block. 

The {\em packing number} $D(v, k)$ is the maximum number of blocks in any $(v, k)$-packing. For $k = 4$, we have the following
result (see \cite{SWY}).

\begin{Theorem}
$D(v, 4) = \lfloor {v\over 4}\lfloor {v-1\over 3}\rfloor\rfloor - \epsilon$, where
\[
\epsilon = \left\{
\begin{array}{cl}
1&{\mbox{if $v\equiv 7$ or $10 \bmod{12}, v\neq 10, 19$}}\\
1&\mbox{if $v = 9$ or $17$}\\
2&\mbox{if $v = 8, 10$ or $11$}\\
3&\mbox{if $v = 19$}\\
0&\mbox{otherwise.}\\
\end{array}
\right.
\]
\end{Theorem}

When every pair of $X$ occurs in exactly one block, the packing is  a {\em balanced incomplete block design} and it is 
denoted as a $(v,k,1)$-BIBD. 

A {\em parallel class} of a packing $\STS$ is a set of disjoint blocks that partitions $X$.
 A \emph{partial parallel class} (or \emph{PPC}) in $\STS$ is any set of disjoint blocks in $\B$. 
 The \emph{size} of a PPC is the number of blocks
in it. A PPC of size $\pc$ is maximum if there does not exist any PPC in $\STS$ of size $\pc + 1$.

Define $\beta (\pc, v, k)$ to be the maximum number of blocks in any $(v, k)$-packing in which the maximum partial parallel class has size
$\pc$. Stinson \cite{stinson} studied $\beta (\pc, v, 3)$ and  gave some explicit constructions
of $(v, 3)$-packings with largest PPC of size $\pc$. 
Upper bounds for $\beta (\pc, v, 3)$ were also proven in \cite{stinson}. 

In this paper, we generalize the results of \cite{stinson} to the case  $k = 4$. We give 
two constructions of $(v, 4)$-packings with largest PPC of size $\pc$ in Section 2.
 Then we provide three upper bounds for $\beta(\pc, v, 4)$ in Section 3. Section 4 determines some exact values of
 $\beta(\pc, v, 4)$.
Some of our results can be generalized to the cases  $k > 4$, as discussed  in Section 5. 

\section{Constructions}\label{s.con}

Stinson \cite{stinson} constructed $(v, 3)$-packings with small maximum PPCs using Room squares. A natural generalization
of that method, to construct 
$(v, 4)$-packings, is obtained by using  Kirkman squares. 

A \emph{Kirkman square}, denoted $KS_k(n)$, is a $t\times t$ array, say $\are$, where $t = \frac{n-1}{k-1}$,  defined on an $n$-set $V$, such that
\begin{enumerate}
\item
every point of $V$ is contained in one cell of each row and column of $\are$,
\item
each cell of $\are$  either is empty or contains a $k$-subset of $V$, and
\item
the collection of blocks obtained from the non-empty cells of $\are$  is an $(n, k,1)$-BIBD.
\end{enumerate}

A $KS_2(n)$ is known as a \emph{Room square}. 
For $k = 3$, \cite{kirkmanSquare} and \cite{kirkmanSquare2} proved the following.

\begin{Theorem}
Let $n$ be a positive integer, $n \equiv 3 \bmod 6$.  Then there exists a $KS_3(n)$ except possibly for  $n \in \{9,15, 21, 141, 153, 165, 177, 189, 231, 249, 261, 285, 351, 357\}$.
\end{Theorem}

Suppose $\are$ is a $KS_3(n)$. Then we can construct an $(n+\pc, 4)$-packing in which the 
largest PPC has  size $\pc$, where $\pc \le \frac{n}{3}$, as follows.
We first select $\pc$ rows of $\are$ such that the first cell in each of these rows is non-empty
(there are $\frac{n}{3}$ such rows). 
Suppose these rows are denoted $r_1, \dots, r_{\pc}$. For each non-empty cell 
in $r_i$, add a new point $s_i$ to form a block of size four. These $\frac{\pc n}{3}$ blocks form an $(n + \pc, 4)$-packing.
The blocks arising from  the cells in the first column of $\are$ form a PPC with size $\pc$. Since every block 
in the packing contains some $s_i$, this PPC is maximum. So we have the following result.

 \begin{Theorem}\label{t.KS}
Let $n$ be a positive integer, $n \equiv 3 \bmod 6$, and \[n \not\in \{ 9, 15,  21, 141, 153, 165, 
177,  189, 231, 249, 261, 285, 351, 357\}.\] Then there is an $(n + \pc, 4)$-packing with $\frac{\pc n}{3}$
blocks, in which the  largest PPC has size $\pc$. 
\end{Theorem}

We can improve this result slightly. If we adjoin the blocks of any $4$-packing on the points $s_1, \dots , s_{\pc}$, then the maximum size of a PPC is still $\pc$, because every block still contains at least one $s_i$.

 \begin{Theorem}\label{t+p.KS}
Let $n$ be a positive integer, $n \equiv 3 \bmod 6$,  \[n \not\in \{ 9, 15,  21, 141, 153, 165, 
177,  189, 231, 249, 261, 285, 351, 357\}.\] Then there is an $(n + \pc, 4)$-packing with $\frac{\pc n}{3} + D(\pc,4)$
blocks, in which the  largest PPC has size $\pc$. 
\end{Theorem}

The construction of  $(v, 4)$-packings in Theorems \ref{t.KS} and \ref{t+p.KS} requires that  $ v - \pc \equiv 3\bmod 6$. 
We now present a modification that relaxes this condition.   
 
  A \emph{transversal design}, denoted $TD(k,n)$, is a triple $(V, \mathcal G, \B)$, where $V$ is a set of $kn$ points, $\mathcal G$
is a partition of $V$ into $k$ \emph{groups}, each of size $n$, and $\B$ is a collection of $k$-subsets of $V$, called \emph{blocks},  such that every
pair of points from $V$ is contained either in one group or in one block, but not both.
A set of $k-2$ mutually orthogonal latin squares (MOLS) of order $n$ is equivalent to a TD$(k, n)$.
For information about MOLS, see \cite{MOLS}. 

A \emph{transversal} of a latin square $L$ of order $n$ is a set of $n$ cells, one from each row and each column of $L$, such that these $n$ cells contain $n$ different symbols. A \emph{common transversal} of a set of $k$ mutually orthogonal latin squares of order $n$ is a set of $n$ cells that is a transversal of each of the $k$ latin squares.

\begin{Theorem}\label{l.RTD}
Suppose there are two MOLS of order $n$ that have a common transversal, and suppose that  
$1 \leq \pc \le n$. Then there is  a $(3n+ \pc, 4)$-packing
with $n\pc + D(\pc,4)$ blocks,
in which the largest PPC has size $\pc$.
\end{Theorem}
\begin{proof}
From the two MOLS of order $n$, we first construct the corresponding TD$( 4, n)$. Then we delete $n-\pc$ points from 
the fourth group of the transversal design and we delete all of the blocks containing these points. The remaining $n\pc$ blocks form a 
$(3n+\pc, 4)$-packing. Finally, we construct a packing of size $D(\pc,4)$ on the $\pc$ non-deleted points in the fourth group, thus obtaining a $(3n+\pc, 4)$-packing with $n\pc + D(\pc,4)$ blocks.

The non-deleted blocks corresponding to the transversal form a PPC of size $\pc$. Since each
block contains at least one point from the $\pc$ non-deleted points in   the fourth group of the $TD$, the maximum size of any PPC is $\pc$.
\end{proof}

A {\em self-orthogonal latin square of order $n$} (denoted SOLS$(n)$) is a latin square that is orthogonal to its transpose. It is easy to see that the main diagonal of a SOLS$(n)$ is a common  transversal of this set of two MOLS$(n)$.

\begin{Theorem}\cite{SOLS}\label{t.sols}
A self-orthogonal latin square of order $n$ exists for all positive integers $n \neq 2, 3 $ or $6$.
\end{Theorem}

From Theorem \ref{l.RTD}  and Theorem \ref{t.sols}, we have the following.

\begin{Theorem}\label{t.con}
If  $n \ge 4, n \neq 6$, then
 for $1 \leq \pc \leq n$,   there is a $(3n + \pc, 4)$-packing with $n\pc + D(\pc,4)$
blocks, in which the size of largest PPC has size $\pc$. 
\end{Theorem}

\begin{Remark}
{\rm 
Theorems \ref{t+p.KS} and \ref{t.con} yield similar results. They both produce packings of size roughly equal to 
\[ \frac{\pc(v-\pc)}{3} + \frac{\pc(\pc - 1)}{12},\] for appropriate values of $v$.\hfill$\blacksquare$
}
\end{Remark}

Theorem \ref{t.con} yields a $(v,4)$-packing with a PPC of size $\pc$, in which $v - \pc \equiv 0\bmod 3$.
We now describe some variations where $v - \pc \equiv 1,2\bmod 3$.

Suppose we use Theorem \ref{t.con} to construct a $(3n + \pc, 4)$-packing with $n\pc + D(\pc,4)$
blocks, in which the size of largest PPC has size $\pc$. Then delete a point $x$ that is not in a block of the PPC of size $\pc$ (this requires $\pc\leq n-1$), along with the $\pc$ blocks that contain $x$. We obtain the following.

\begin{Theorem}\label{t2.con}
If  $n \ge 4, n \neq 6$, then
 for $1 \leq \pc \leq n-1$,   there is a $(3n-1 + \pc, 4)$-packing with $(n-1)\pc + D(\pc,4)$
blocks, in which the size of largest PPC has size $\pc$. 
\end{Theorem}

In Theorem \ref{t2.con}, we construct a $(v,4)$-packing with a PPC of size $\pc$, in which $v - \pc \equiv 2\bmod 3$. We now handle the last case, where $v - \pc \equiv 1\bmod 3$.
Again, we use Theorem \ref{t.con} to construct a $(3n + \pc, 4)$-packing with $n\pc + D(\pc,4)$
blocks, in which the size of largest PPC has size $\pc$. We carry out the following modifications:
\begin{enumerate}
\item Delete the $D(\pc,4)$ blocks on the $\pc$ points in the last group.
\item Adjoin the blocks of a  $4$-packing consisting of $D(\pc +1,4)$ blocks constructed on the the $\pc$ points in the last group and one new point.
\end{enumerate}

\begin{Theorem}\label{t1.con}
If  $n \ge 4, n \neq 6$, then
 for $1 \leq \pc \leq n$,   there is a $(3n+1 + \pc, 4)$-packing with $n\pc + D(\pc+1,4)$
blocks, in which the size of largest PPC has size $\pc$. 
\end{Theorem}

\section{Upper bounds}

In this section, we prove three upper bounds on $\beta(\pc,v,4)$ (i.e., the maximum number of blocks in a $(v, 4)$-packing with largest PPC of size $\pc$).

We begin by defining some notation.
Suppose $\STS = (X, \B)$ is a $(v,4)$-packing with $b$ blocks, in which the 
maximum partial parallel class has size  $\pc$; hence $v \geq 4\pc$.
Let $\Par = \{B_1, \dots , B_\pc\}$ be a set of $\pc$ disjoint blocks 
and let $P = \bigcup_{i=1}^{\pc}  B_i$. Let $T = X \setminus P$. 
Thus, $|P| = 4 \pc$ and $|T| = v - 4\pc$.

For $x\in P$ and for $0 \leq i \leq 3$, let $\mathcal{T}_x^i$ denote the set of blocks not in $\Par$
that contain the point $x$ and exactly $i$ points in $T$.   Define $t_x^i = |\mathcal{T}_x^i|$. 

\subsection{The First Bound}

Our first bound uses the method described in \cite{stinson}. We observe that, for any two points $x, y \in B_i \in \Par$, a block  $B \in \mathcal{T}_x^3$ and a block  $B' \in \mathcal{T}_y^3$ cannot be disjoint. Otherwise, we can delete $B_i$ and add the two blocks $B$ and $B'$ to the PPC to get a new PPC of size $\pc + 1$.  

The 
following lemma is a  straightforward consequence of this observation.

\begin{Lemma}
\label{l-t_w}
Suppose $B_i  = \{w,x,y,z\} \in \Par$ and $t_w^3 \geq \max \{  t_x^3,t_y^3,t_z^3\}$. Then  one of  the following two conditions holds:
\begin{enumerate}
\item $t_w^3 \leq 3$, or
\item $t_w^3 \geq 4$ and  $t_x^3 = t_y^3 = t_z^3 = 0$.
\end{enumerate}
\end{Lemma}

\begin{proof}
Suppose $t_w^3 \geq 4$ and $B \in \mathcal{T}_x^3 \cup  \mathcal{T}_y^3 \cup \mathcal{T}_z^3 $. Then there is a block $B' \in \mathcal{T}_w^3$ such that
$B \cap B' = \emptyset$. This contradicts the observation above. 
Hence $t_x^3 = t_y^3 = t_z^3 = 0$ if  $t_w^3 \geq 4$. 
\end{proof}

 Since $\Par$ contains $\pc$ blocks, we have the following bound.
\begin{Theorem}\label{t.bound1}
\[ \beta (\pc,v,4) \leq 
\pc \left( (8\pc -7) + \max \left\{ 12, \left\lfloor \frac{v-4\pc}{3} \right\rfloor \right\}\right).
\]
\end{Theorem}
\begin{proof}
Since $\Par$ is maximum, there is no block contained in $T$. 
From Lemma \ref{l-t_w}, 
there are at most $\pc \times \max \left\{ 12, \left\lfloor \frac{v-4\pc}{3} \right\rfloor \right\}$
blocks having one point in $P$ and three points in $T$. 

The number of pairs of points in $P$ that are not contained in a block of $\Par$ is 
\[ \binom{4\pc}{2}  - 6\pc = 8\pc (\pc -1).\]  Therefore there are at most $8\pc (\pc -1)$ 
blocks not in $\Par$ that contain at least two points in $P$. Finally, there are $\pc$ blocks in $\Par$.

Therefore, in total,  there are at most 
\[8\pc (\pc -1) + \pc  \times \max \left\{ 12, \left\lfloor \frac{v-4\pc}{3} \right\rfloor \right\} + \pc\] blocks in the packing.
\end{proof}

\begin{Corollary}
\label{cor-bound}
For $v \geq 4 \pc + 36$, it holds that 
\[ \beta (\pc,v,4) \leq 
\frac{\pc v}{3} + \frac{20 \pc^2}{3}- 7 \pc.
\]
\end{Corollary}

\begin{proof} 
We have that  
\[ \frac{v-4\pc}{3} \geq 12\]
if and only if $v \geq 4 \pc + 36$. Thus, when $v \geq 4 \pc + 36$, Theorem \ref{t.bound1} yields
\[ \beta (\pc,v,4) \leq 
\pc \left( (8\pc -7) + \frac{v-4\pc}{3} \right) =  \frac{\pc v}{3} + \frac{20 \pc^2}{3}- 7 \pc.
\]
\end{proof}

\subsection{The Second Bound}

We now prove a somewhat better bound by using a more precise counting argument.
The next two lemmas are straightforward.

\begin{Lemma}
For any $x \in P$, the following two inequalities hold:
\begin{equation}
\label{eq1}
3t_x^3 + 2t_x^2 + t_x^1 \leq v - 4 \pc.
\end{equation}
and
\begin{equation}
\label{eq2}
t_x^2 + 2t_x^1 + 3t_x^0 \leq 4 (\pc - 1).
\end{equation}
\end{Lemma}

\begin{Lemma}
\label{l3}
The number of blocks $b$ in the packing is given by the following formula:
\begin{equation}
\label{eq3}
b = \pc + \sum_{x \in P} t_x^3 + \frac{1}{2}\sum_{x \in P}t_x^2 + \frac{1}{3}\sum_{x \in P}t_x^1 
+ \frac{1}{4}\sum_{x \in P}t_x^0. 
\end{equation}
\end{Lemma}

For any $x \in P$, define 
\[c_x =  t_x^3 + \frac{t_x^2}{2} + \frac{t_x^1}{3} + \frac{t_x^0}{4} 
= \frac{12t_x^3 + 6t_x^2 + 4t_x^1 + 3t_x^0}{12}.\]
Then it is clear from Lemma \ref{l3} that the following equation holds:
\begin{equation}
\label{eq4}
b = \pc + \sum_{x \in P} c_x. 
\end{equation}

Our strategy is to obtain upper bounds on $c_x$ given the constraints (\ref{eq1}) and (\ref{eq2}). This leads to an integer program; however, for convenience,  we will consider the linear programming relaxation. 
For ease of notation, let us fix a point $x$ and denote $y_3 = t_x^3$,  $y_2 = t_x^2$, $y_1 = t_x^1$ and $y_0 = t_x^0$.
We are interested in the optimal solution to the following LP:

\begin{center}
\fancybox{ \vspace{-.2in}
\begin{alignat}{3}
&&& \nonumber \text{maximize } && 12y_3 + 6y_2 + 4y_1 + 3y_0 \\ 
&&& \nonumber \text{subject to the constraints}\\ 
&&&\label{LP1} 3y_3 + 2y_2 + y_1 \leq v - 4 \pc \\ 
&&&\label{LP2} y_2 + 2y_1 + 3y_0  \leq 4(\pc - 1)\\ 
&&&\nonumber y_3, y_2, y_1, y_0 \geq 0
\end{alignat}
}
\end{center}

\medskip

If we compute $4 \times (\ref{LP1}) + (\ref{LP2})$, we obtain the following bound:
\begin{equation}
\label{LP3} 12y_3 + 9y_2 + 6y_1 + 3y_0 \leq 4v - 12\pc -4.
\end{equation}
Since
\[  12c_x = 12y_3 + 6y_2 + 4y_1 + 3y_0 \leq 12y_3 + 9y_2 + 6y_1 + 3y_0, \] 
we have the following.
\begin{Lemma}
\begin{equation}
\label{LP4} c_x  \leq \frac{v - 3\pc -1}{3}.
\end{equation}
\end{Lemma}

\medskip

\begin{Remark}
{\rm We note that we can achieve equality in (\ref{LP4}) by taking
\[y_3  = \frac{v - 4\pc}{3}, \quad 
y_2 = 0, \quad 
y_1 = 0, \quad \text{and} \quad
y_0 = \frac{4(\pc-1)}{3}.
\]
Thus, the optimal solution to the LP is $4(v - 3\pc -1)$.\hfill$\blacksquare$
}
\end{Remark}

\medskip

We are also interested in the optimal solution to the LP in the special cases where $y_3 \leq 3$. 
Here, we just use the inequality \[6y_2 + 12y_1 + 18y_0 \leq 24(\pc - 1),\] which follows immediately from (\ref{LP2}).
Since \[ 12c_x = 12y_3 + 6y_2 + 4y_1 + 3y_0 \leq 12y_3 + 6y_2 + 12y_1 + 18y_0, \] we obtain
the following.
\begin{Lemma}
\label{lem5}
\begin{equation}
\label{LP5} c_x  \leq y_3 + 2\pc - 2.
\end{equation}
\end{Lemma}
Since we are assuming that $y_3 \leq 3$, we have
\begin{equation}
\label{LP6} c_x  \leq 2 \pc +1.
\end{equation}

\medskip

\begin{Remark}
{\rm We can achieve equality in (\ref{LP5}) by taking
\[
y_2 = 4(\pc - 1), \quad 
y_1 = 0, \quad \text{and} \quad
y_0 = 0.
\]

This is a feasible solution to the LP provided that (\ref{LP1}) is satisfied, i.e., if 
\[ 3y_3 + 8(\pc-1) \leq v - 4 \pc,\]
which simplifies to 
\[v \geq 3y_3 + 12 \pc - 8.\]
We are assuming that $y_3 \leq 3$, so the optimal solution to the LP 
is $24 \pc + 12$ whenever $v \geq 12 \pc +1$.\hfill$\blacksquare$
}
\end{Remark}

\medskip

The following lemma is an immediate application of (\ref{LP6}).

\begin{Lemma}
\label{lem3.10}
Suppose that %$v \geq 12 \pc +1$ and 
$B_i  = \{w,x,y,z\} \in \Par$ and $\max \{ t_w^3, t_x^3,t_y^3,t_z^3\} \leq 3$. Then 
\begin{equation}
\label{LP7}
c_w  + c_x  + c_y  + c_z  \leq  8 \pc + 4.
\end{equation}
\end{Lemma}

\begin{Lemma}
\label{lem3.11}
Suppose that $B_i  = \{w,x,y,z\} \in \Par$ and $\max \{ t_w^3, t_x^3,t_y^3,t_z^3\} \geq 4$. Then 
\begin{equation}
\label{LP8}
c_w  + c_x  + c_y  + c_z  \leq \frac{v}{3} + 5 \pc -\frac{19}{3}.
\end{equation}
\end{Lemma}

\begin{proof}
Without loss of generality, assume that $t_w^3 = \max \{ t_w^3, t_x^3,t_y^3,t_z^3\} \geq 4$.
Then $t_x^3 = t_y^3 =t_z^3 = 0$. Hence, we have
$c_w \leq (v - 3\pc -1)/3$ from (\ref{LP4}) and we obtain 
$c_x, c_y, c_z \leq 2\pc - 2$ by setting $y_3 = 0$ in  (\ref{LP5}). Hence,
\[ c_w  + c_x  + c_y  + c_z  \leq 3(2\pc - 2) + \frac{v - 3\pc -1}{3} = \frac{v}{3} + 5 \pc -\frac{19}{3}.\]
\end{proof}

\begin{Theorem}
\label{bound-improved}
Suppose $v \geq 9 \pc + 31$. Then
\[ \beta (\pc,v,4) \leq \frac{\pc v}{3} + 5 \pc^2 -\frac{16\pc}{3}.\]
\end{Theorem}

\begin{proof}
If $v \geq 9 \pc + 31$, then 
\[ \frac{v}{3} + 5 \pc -\frac{19}{3} \geq 8 \pc + 4.\]
Hence, from Lemmas \ref{lem3.10} and \ref{lem3.11}, 
\[c_w  + c_x  + c_y  + c_z  \leq \frac{v}{3} + 5 \pc -\frac{19}{3}\]
for all $\pc$ blocks $\{w,x,y,z\} \in \Par$.
Now, applying (\ref{eq4}), we obtain the upper bound
\[ \beta (\pc,v,4) \leq \pc + \pc \left( \frac{v}{3} + 5 \pc -\frac{19}{3} \right) 
= \frac{\pc v}{3} + 5 \pc^2 -\frac{16\pc}{3}.\]
\end{proof}

\begin{Remark}
{\rm 
Ignoring lower order terms, the upper bound on $\beta (\pc,v,4)$ proven in Theorem \ref{bound-improved} is 
\[\frac{\pc v}{3} + 5 \pc^2,\] while the previous bound from Corollary \ref{cor-bound} was 
\[\frac{\pc v}{3} + \frac{20 \pc^2}{3}.\]
\hfill$\blacksquare$
}
\end{Remark}

%\newpage

\subsection{The Third  Bound}

The third upper bound on $\beta( \rho, v, 4)$ is based on more refined analysis of $\mathcal{T}_x^i, 0\le i \le 3$. 
As defined above, the blocks in the PPC are denoted as $B_i = \{ a_i, b_i, c_i, d_i\},$ for $i = 1, \dots, \rho$. We further 
assume that 
\[t_{a_1}^3 \le
t_{a_2}^3 \le \cdots \le t_{a_\rho}^3\] and  \[t_{a_i}^3 \ge \max \{ t_{b_i}^3,  t_{c_i}^3, t_{d_i}^3\},\]
for $i = 1,\dots, \rho$. Let $A = \{ a_i: 1\le i \le \rho \}$.

Now we will partition the blocks of $\mathcal B \setminus \mathcal P$ into various subsets as follows.
\begin{enumerate}
\item
$\{\mathcal{T}^3_x: x \in P\}$.
\item
For blocks in $\bigcup_{x \in P} \mathcal{T}_x^2$, let 
\begin{eqnarray*}
\mathcal{A}_i &=& \{ \{ a_i, e, y, z\} \in \mathcal B\setminus \mathcal{P}: e \in P\setminus A; \; y, z \in T\}\\
\mathcal{A}_i' &=& \{ \{ a_i, a_s, y, z \} \in \mathcal B\setminus  \mathcal{P}: i+1 \le s \le \rho; \; y, z \in T \}\\ 
\mathcal{C} &=& \{ \{ e, f, y, z \} \in \mathcal B \setminus  \mathcal{P} : e, f \in P\setminus A;\;  y, z \in T \}.
\end{eqnarray*}
The blocks in $\mathcal{A}_i$ contain one point in $A$, the blocks in $\mathcal{A}'_i$ contain two points in $A$, and the  blocks in $\mathcal{C}$ contain no points in $A$. 
\item
For blocks in $\bigcup_{x \in P} \mathcal{T}_x^1$, let 
\begin{eqnarray*}
\mathcal{E}   &=& \{ \{ e, f, g, z \} \in \mathcal B\setminus \mathcal P: 
e, f , g\in P\setminus A;\;  
z \in T \},
\end{eqnarray*} 
and let $\mathcal{E}'$ consist of the remaining blocks  in $\bigcup_{x \in P} \mathcal{T}_x^1$ 
 (the blocks in $\mathcal{E}$ contain no points in $A$ and the blocks in $\mathcal{E}'$ contain at least one point in $A$).
 Further, we partition the blocks in $\mathcal{E}'$ into subsets $\mathcal{E}'_1, \dots , \mathcal{E}'_{\rho}$, where a block in $B\in \mathcal{E}'$ is placed in $\mathcal{E}'_i$ if $a_i \in B$ and $a_j \not\in B$ for any $j < i$.
\item
 For blocks in $\bigcup_{x \in P} \mathcal{T}_x^0$, let 
\begin{eqnarray*}
\mathcal{F}   &=& \{ \{ e, f, g, h \} \in \mathcal B\setminus P: e, f, g, h
  \in P\setminus A\}, 
  \end{eqnarray*}
 and let $\mathcal{F}'$ consist of the remaining blocks  in  $\bigcup_{x \in P} \mathcal{T}_x^0$ (the blocks in $\mathcal{F}$ contain no points in $A$ and the blocks in $\mathcal{F}'$ contain at least one point in $A$).  Further, we partition the blocks in $\mathcal{F}'$ into subsets $\mathcal{F}'_1, \dots ,\mathcal{F}'_{\rho}$, where a block in $B\in \mathcal{F}'$ is placed in $\mathcal{F}'_i$ if $a_i \in B$ and $a_j \not\in B$ for any $j < i$.
\end{enumerate}

For future use, we define the following notation:
\begin{eqnarray*}
\alpha_i & = & |\mathcal{A}_i| + |\mathcal{A}'_i|, \\
\epsilon_i & = & |\mathcal{E}'_i|, \quad \text{and}  \\
\zeta_i & = & |\mathcal{F}'_i| ,
\end{eqnarray*}
for $i = 1, \dots , \rho$.

 Since the various subsets of blocks defined above are disjoint,  the total number of  blocks
 in the packing is 
\begin{equation}
\label{total.eq}
b = \sum_{x \in P}|\mathcal{T}_x^3| + \sum_{i = 1}^ \rho (\alpha_i + \epsilon_i + \zeta_i) + |\mathcal{C}| +|\mathcal{E}| +  |\mathcal{F}|   + \rho.
\end{equation}

  \begin{Lemma}
  For $i = 1, \dots, \rho$, it holds that
  \begin{equation} \alpha_i + 2 \epsilon_i + 3 \zeta_i \le 4\rho -i - 3,\label{f.C}.
  \end{equation}
  \begin{proof}
  Let $1 \leq i \leq \rho$.  Denote $U_i = P \setminus ( A_i \cup \{a_j : j < i \} )$. We note that $|U_i|= 4\rho -i - 3$.
  Each block in $|\mathcal{A}_i| \cup |\mathcal{A}'_i|$ contains one point in $U_i$, each block in $|\mathcal{E}'_i|$ contains two points from $U_i$ and each block in $|\mathcal{F}'_i|$ contains three points from $U_i$. Further, each point from $U_i$ occurs in at most one block of $|\mathcal{A}_i| \cup |\mathcal{A}'_i| \cup
 |\mathcal{E}'_i| \cup |\mathcal{F}'_i|$. The result follows.
  \end{proof}
  \end{Lemma}
  
  \begin{Lemma}
 For $i = 1, \dots, \rho$, it holds that
 \begin{equation} t_{a_i}^3 + \alpha_i +  \epsilon_i +  \zeta_i \le 
 \frac{v-i-3}{3}.
 \label{f.cor}
  \end{equation}
    \end{Lemma}
 \begin{proof} 
 Let $1 \leq i \leq \rho$. Every block  in $\mathcal{T}_{a_i}^3 \cup |\mathcal{A}_i| \cup |\mathcal{A}'_i| \cup
 |\mathcal{E}'_i| \cup |\mathcal{F}'_i|$  contains the point $a_i$ but these blocks are otherwise disjoint. Also, none of these blocks contains $b_i, c_i$ or $d_i$, or any $a_j$ with $j < i$. The result follows.
    \end{proof}

  \medskip

  For $1 \leq i \leq \rho$, 
  denote $L_i = t_{a_i}^3 +t_{b_i}^3+t_{c_i}^3+ t_{d_i}^3+  \alpha_i + \epsilon_i + \zeta_i $.
 
\begin{Lemma} \label{r.1}
 Suppose $v \ge 12\rho + 28$. Then for each $i, 1\le i \le \rho$, we have 
 \begin{eqnarray}\label{f.1}
 L_i \le {v-i -3\over 3}.
 \end{eqnarray}
 \end{Lemma}
 \begin{proof}
 
 First, consider the case when $t_{a_i}^3 \ge 4$. Then, $t_{b_i}^3 + t_{c_i}^3 + t_{d_i}^3 = 0$ and
 we have 
\begin{eqnarray*} 
L_i &=&  t_{a_i}^3 +  \alpha_i + \epsilon_i + \zeta_i \\
&\leq & \frac{v - i - 3 }{3}
 \end{eqnarray*} 
 from (\ref{f.cor}).

 Now we consider the case where $t_{a_i}^3 \le 3$.
 Then we have 
  \begin{eqnarray*} 
L_i &\leq& 12 + \alpha_i +  \epsilon_i +  \zeta_i \\
&\leq& 12 + \alpha_i + 2 \epsilon_i + 3 \zeta_i \\
&\leq & 12 + 4\pc - i -3\\
& = & 4\pc - i +9,
 \end{eqnarray*} 
 from (\ref{f.C}).
 Since $v \ge 12\rho + 28$, we have $\pc \leq \frac{v-28}{12}$
 and hence
   \begin{eqnarray*} 
L_i &\leq& \frac{v-28}{3} - i + 9  \\
&=& \frac{v-3i-1}{3} \\
&\leq & \frac{v-i-3}{3}, 
 \end{eqnarray*} 
 since $i \geq 1$.
\end{proof}

% We note that, when $L_i = {v - i - 3\over 3}$, we must have $|A_i| + |A_i'| = 4\rho - i -3$. We will use this fact later. 

\begin{Lemma}\label{l.emptyC}
If $\pc\geq 2$, $t_{a_1}^3 \ge 3$ and $t_{a_2}^3 \ge 6$, then $\mathcal{C} = \emptyset$.
\end{Lemma} 

\begin{proof}
First, suppose there is a block $B = \{ e, f, x, y\} \in \mathcal{C}$, where $x, y \in T$ and $e, f \in P
\setminus A$. Assume that $e \in B_i$ and $f \in B_j$, where $i < j$. Since $t_{a_1}^3 \le
t_{a_2}^3 \le \cdots \le t_{a_\rho}^3,$ we have $t_{a_i}^3 \ge 3$ and $t_{a_j}^3 \ge 6$.

Since $t_{a_i}^3 \ge 3$, we can choose a block $B'_1 \in \mathcal{T}_{a_i}^3$ that 
is disjoint from $B$ (we just choose $B'_1 \in \mathcal{T}_{a_i}^3$ such that  $x, y \not\in B$).
Similarly, since $t_{a_j}^3 \ge 6$, we can choose a block 
$B_2' \in \mathcal{T}_{a_j}^3$ that is disjoint from $B$ and $B_1'$ (note that $B$ and $B_1'$ contain five points from $T$).
Then, deleting $B_i$ and $B_j$ from the PPC and  adjoining $B$, $B_1'$ and $B_2'$,  we
obtain $\rho + 1$ disjoint blocks, which is a contradiction. 
\end{proof}

\begin{Lemma}\label{l.emptyE}
If $\pc\geq 3$, $t_{a_1}^3 \ge 2$, $t_{a_2}^3 \ge 5$ and  $t_{a_3}^3 \ge 8$, then $\mathcal{E} = \emptyset$.
\end{Lemma} 

\begin{proof}
Suppose there is a block $B = \{ b, c, d, x\} \in \mathcal{E}$, where $b, c, d  \in P\setminus A $, $b \in B_i$, $c \in B_j$, $d \in B_k$, and $x \in T$. Similar to the proof of Lemma \ref{l.emptyC}, 
 we can find $B_1' \in \mathcal{T}_{a_i}^3$, $B_2' \in \mathcal{T}_{a_j}^3$ and $B_3' \in \mathcal{T}_{a_k}^3$ such that $B, B'_1,   B_2', B_3'$ are disjoint. If we delete $B_i, B_j$ and $B_k$ from the PPC and adjoin $B, B'_1,   B_2'$ and $B_3'$,
we obtain $\rho + 1$ disjoint blocks, which is a contradiction.
\end{proof}

\begin{Lemma}\label{l.emptyF}
If $\pc\geq 4$, $t_{a_1}^3 \ge 1$, $t_{a_2}^3 \ge 4$, $t_{a_3}^3 \ge 7$ and $t_{a_4}^3 \ge 10$, then $\mathcal{F} = \emptyset$.
\end{Lemma} 

\begin{proof}
Suppose there is a block $B = \{ b, c, d, e\} \in \mathcal{F}$, where $b, c, d ,e \in P\setminus A $, $b \in B_i$, $c \in B_j$, $d \in B_k$, and $e \in B_{\ell}$. Similar to the proof of Lemma \ref{l.emptyC}, we can find $B_1' \in \mathcal{T}_{a_i}^3$, $B_2' \in \mathcal{T}_{a_j}^3$, $B_3' \in \mathcal{T}_{a_k}^3$ and $B_4' \in \mathcal{T}_{a_{\ell}}^3$ such that $B, B'_1,   B_2', B_3', B_4'$ are disjoint. If we delete $B_i, B_j$, $B_k$ and $B_{\ell}$ from the PPC and adjoin $B, B'_1,   B_2', B_3'$ and $B_4'$, 
we obtain $\rho + 1$ disjoint blocks, which is a contradiction.
\end{proof}

 \begin{Lemma}\label{l.nC}
 Suppose $\rho \geq 4$, $t_{a_1}^3 \ge 3$, $t_{a_2}^3 \ge 6$, $t_{a_3}^3 \ge 8$ and $t_{a_4}^3 \ge 10$, and $v \ge 12\rho + 28$. Then the number of blocks  in the packing is at most 
\[
{\rho v \over 3} -  { \rho(\rho + 1)\over 6}  
\]
\end{Lemma}
\begin{proof}
 Let $M = \sum_{i = 1}^\rho L_i$. From Lemma \ref{r.1}, we have 
   \begin{eqnarray*}
 M&\le&\sum_{i = 1}^\rho { ( v - i - 3)\over 3}  \\
 &=&{\rho  ( v - 3)\over 3} - {\rho(\rho +1)\over 6} \\
 &=&{ \rho v\over 3}  - { \rho^2+ 7\rho\over 6}.
 \end{eqnarray*}
  Since $\mathcal{C}\cup \mathcal{E}  \cup \mathcal{F} = \emptyset$ by Lemmas \ref{l.emptyC}, \ref{l.emptyE} and \ref{l.emptyF},   the total number of blocks is 
\begin{eqnarray*}
M + \rho  &\leq & { \rho v\over 3}  - { \rho^2+ 7\rho\over 6}  + \rho \\
& = & { \rho v\over 3}  - {\rho( \rho+ 1)\over 6}.
\end{eqnarray*}
\end{proof}
  
Lemma \ref{l.nC} provides a good upper bound on the number of blocks when the four smallest values $t_{a_i}$ are large enough, because the conditions ensure there are no blocks in 
$\mathcal{C}\cup \mathcal{E}  \cup \mathcal{F}$. On the other hand, if even one of these four values is ``small,'' then we will obtain a bound on the number of blocks by upper-bounding the relevant $t_{a_i}$ by a quantity that is independent of $v$. In this situation, we will just use a trivial upper bound on the number of blocks in $\mathcal{C}\cup \mathcal{E}  \cup \mathcal{F}$.

%From Lemma \ref{l.emptyC} we have

\begin{Lemma}\label{l.C}
Suppose $t_{a_1}^3 < 3$, $t_{a_2}^3 < 6$, $t_{a_3}^3 < 8$ or $t_{a_4}^3 < 10$.
If   $v \ge 12\rho +28$,
 then the number of blocks in   the packing is at most
\[
 { (\rho -1)v\over 3} +  {\rho(13\rho- 2)\over 3}   + 10
\]
\end{Lemma}
\begin{proof}
Let $M = \sum_{j = 1}^\rho L_j$. Since   $v \ge 12\rho + 28$, we can apply Lemma \ref{r.1}.
Choose $i \le 4$ such that $t_{a_i}^3 \leq 3i-1$ (at least one such value of $i$ exists).
 
 Suppose first that  $ t_{a_i}^3\ge 4$. Then  $t_{b_i}^3 = t_{c_i}^3 =t_{d_i}^3 =0$.   We have $L_j \leq {v -j - 3\over 3}$ for all $j$ from (\ref{f.1}). Also, 
 $L_i \leq 3i - 1 + 4\rho - i - 3$ from (\ref{f.C}).
Thus we have
 \begin{eqnarray*}
 M&\le&\sum_{j = 1, j \neq i}^{\rho}\left( 
  {v -j - 3\over 3}\right) +3i -1 +4\rho - i -3\\
 &=&{(\rho - 1)  (v - 3)\over 3} - {\rho(\rho +1)\over 6} + {i\over 3}+ 2i  +4\rho -4\\
 &\le &{(\rho - 1)  (v - 3)\over 3} - {\rho(\rho +1)\over 6} + {28\over 3} +4\rho -4 
 \quad \quad \text{since $i \leq 4$}\\
 &=& {(\rho - 1)  (v - 3)\over 3} - {\rho(\rho -23)\over 6} + {16\over 3}.
 % &<& {(\rho - 1)  (v - 3)\over 3} - {\rho(\rho -37)\over 6} + 7
 \end{eqnarray*}
 On the other hand, if $t_{a_i} \le 3$, then $t_{a_i} + t_{b_i}+ t{c_i} + t_{d_i} \le 12$. So, by a similar argument, 
 we obtain
 \begin{eqnarray*}
 M&\le&\sum_{j = 1, j \neq i}^{\rho}\left( 
  {v -j - 3\over 3}\right) +12 +4\rho - i -3\\
 &=&{(\rho - 1)  (v - 3)\over 3} - {\rho(\rho +1)\over 6} + {i\over 3 } - i  +4\rho + 9\\
 &< &{(\rho - 1)  (v - 3)\over 3} - {\rho(\rho +1)\over 6} + 4\rho + 9
 \quad \quad \text{since $i > 0$}\\
 &=& {(\rho - 1)  (v - 3)\over 3} - {\rho(\rho -23)\over 6} + 9.
 %&<&{(\rho - 1)  (v - 3)\over 3} - {\rho(\rho -37)\over 6} + 7.
 \end{eqnarray*}
 
 Now, each block in $\mathcal{C}\cup \mathcal{E} \cup \mathcal{F}$ contains at least one pair of points from $P \setminus A$. Further,  none of these blocks contains more than one point from any block in $\mathcal{P}$.
Hence, \[|\mathcal{C}\cup \mathcal{E} \cup \mathcal{F}| \le {3\rho \choose 2} -3\rho = {9\rho (\rho -1)\over 2}.\]
Therefore, the total number of  blocks, $b$, satisfies the following inequality: 
 \begin{eqnarray*}
b &\le&  {(\rho - 1)  (v - 3)\over 3} - {\rho(\rho -23)\over 6}  +9 + {9\rho (\rho -1)\over 2} + \rho\\
 &=&{ (\rho -1)v\over 3} - (\rho -1)-  {\rho(\rho -23)\over 6} +9 + {9\rho (\rho -1)\over 2} + \rho \\
 &=& { (\rho -1)v\over 3} -  {\rho(\rho -23)\over 6} + {9\rho (\rho -1)\over 2}  + 10\\
 &=&{ (\rho -1)v\over 3} +  {\rho(13\rho -2)\over 3}  + 10.
  \end{eqnarray*}
 
 \end{proof}

 When $v$ is sufficiently large compared to $\rho$, the bound of Lemma \ref{l.nC} is the relevant bound.

 \begin{Theorem}\label{t.bound}
 Suppose $v \ge {1\over 2}  (27\rho^2 - 3\rho + 60)$. Then the number of
 blocks in  the packing is at most 
 \[
  {\rho v  \over 3} - {  \rho(\rho + 1)\over 6}.
 \] 
 \end{Theorem} 
\begin{proof}
From Lemmas \ref{l.nC} and \ref{l.C}, we have
\[b \leq  \max \left\{ {\rho v  \over 3} - {  \rho(\rho + 1)\over 6}, {(\rho -1)v\over 3} +  {\rho(13\rho- 2)\over 3}   + 10 \right\} .\]
Since $v \ge {1\over 2}(  27\rho^2 - 3\rho + 60)$, we have 
\begin{eqnarray*}
{\rho v  \over 3} - {  \rho(\rho + 1)\over 6} - \left( { (\rho -1)v\over 3} +  {\rho(13\rho- 2)\over 3}   + 10\right )
%&=&{v\over 3}  - {  \rho(\rho + 1)\over 6} -  {\rho(13\rho -2)\over 3}   - 10\\
&=&{v\over 3}  -    {\rho(27\rho - 3)\over 6}   -10 \\
&\ge & 0.
\end{eqnarray*}
 \end{proof}

\section{Some Values of $\beta(\rho, v,4)$}

In this section, we  determine some exact values of $\beta(\rho, v, 4)$.
First, for $\rho = 1$, we can determine the exact values of $\beta (1, v, 4)$ for all $v$.

 \begin{Theorem}
 \[
 \beta (1, v, 4) = \left\{
 \begin{array}{cl}
 D(v ,4) &\mbox{ if $ 4 \le v \le 13$ }\\
 13&\mbox{ if $14\le v \le 39$}\\
 \lfloor {v -1 \over 3}\rfloor& \mbox{ if $v \ge 40$}
 \end{array}
 \right.
 \] 
 \end{Theorem}  
\begin{proof}
When $4\le v \le 6$, $D(v, 4) = 1$, so $\beta(1, v, 4) = 1$. For $7 \le v \le 11$, we display the blocks of the optimal packings in Table \ref{tb.1}. These packings do not contain any disjoint blocks. For $v =12$, we have $D(12,4) = 9$ and the optimal packing is obtained by deleting a point $x$ and the four blocks containing $x$ from a projective plane of order $3$. This packing also does not contain disjoint blocks.  For $v =13$, we have $D(13,4) = 13$ and the optimal packing is a projective plane of order $3$, which does not contain disjoint blocks.

\begin{table}
\caption{Some small $(v,4)$-packings}\label{tb.1}
\[
\begin{array}{|c|c|cccccc|}
\hline
 v&D(v,4)&&&&&&\\
 \hline
 7&2&1,2,3,4&1,5,6,7&&&&\\
 \hline
 8&2&1,2,3,4&1,5,6,7&&&&\\
 \hline
 9&3& 1,2,3,4&1,5,6,7&2,5,8,9&&&\\ 
 \hline 
 10&5&1,2,3,4&1,5,6,7&2,5,8,9& 3,6,8,10&4,7,9,10&\\
 \hline
11&6&1,2,3,4&1,5,6,7&1,8,9,10&2,5,8,11&3,6,9,11&4,7,10,11\\
\hline
%12&9&1,2,3,4&1,5,6,7&1,8,9,10&2,5,8,11&3,6,9,11&4,7,10,11\\
%&&2,6,10,12&3,7,8,12&4,5,9,12&&&\\
%\hline
%13&13&1,2,3,4&1,5,6,7&1,8,9,10&2,5,8,11&3,6,9,11&4,7,10,11\\
%&&2,6,10,12&3,7,8,12&4,5,9,12&4,6,8,13&2,7,9,13&3,5,10,13\\
%&&1,11,12,13&&&&&\\
%\hline
\end{array}
\]
\end{table}

For $13\leq v\leq 40$, Theorem \ref{t.bound1} 
gives the bound $\beta(1,v,4) \leq 13 $, and  
for $v \ge 40$, Theorem \ref{t.bound1} 
gives the bound $\beta(1,v,4) \leq\lfloor {v - 1\over 3}\rfloor$.
For $v \geq 40$, Theorems \ref{t.con}, \ref{t2.con} and \ref{t1.con} provide the desired packings. 
For $14\le v \le 39$, 
 $\beta( 1, v, 4) = 13$ because 
 \[13 = \beta (1, 13, 4) \leq \beta (1, v, 4) \le \beta(1, 40, 4)  = 13.\]
\end{proof}

For other small values of $\rho$, our constructions also give packings in which the number of blocks is
very close to the upper  bound from Theorem \ref{t.bound}.

\begin{Theorem}\label{t.beta2}
Suppose $v \ge 81$.
 Then  
 \[ \beta(2, v, 4) \leq \left\lfloor {2v- 3\over 3} \right\rfloor.\]
   Also, \[ \beta(2, v, 4) \ge \begin{cases}
 {2v-6\over 3} &\mbox{ if $v \equiv 0 \bmod 3$}\\
 {2v-8\over 3} &\mbox{ if $v \equiv 1 \bmod 3$}\\
 {2v-4\over 3} &\mbox{ if $v \equiv 2 \bmod 3$}.
\end{cases}
\]
\end{Theorem}
\begin{proof}
The upper bound follows from Theorem \ref{t.bound}. The lower bounds follow from 
Theorems \ref{t.con}, \ref{t2.con} and \ref{t1.con}.
\end{proof}

 \begin{Theorem}\label{t.beta3}
Suppose $v \ge 147$.
 Then      $ \beta(3, v, 4) \leq v-2$.
Also, \[ \beta(3, v, 4) \ge 
\begin{cases}
 v - 3 &\mbox{ if $v \equiv 0,1 \bmod 3$}\\
 v - 5 &\mbox{ if $v \equiv 2 \bmod 3$}.
\end{cases}
\]
\end{Theorem}
\begin{proof}
The upper bound follows from Theorem \ref{t.bound} and the lower bounds follow from 
Theorems \ref{t.con}, \ref{t2.con} and \ref{t1.con}. Note that we use the fact that $D(4,4) = 1$
when $v \equiv 1 \bmod 3$; in this case, we apply Theorem \ref{t1.con}.
\end{proof}

\begin{Theorem}\label{t.beta4}
Suppose $v \ge 240$.
 Then  
 \[ \beta(4, v, 4) \leq \left\lfloor {4v -10\over 3} \right\rfloor.\]
   Also, \[ \beta(4, v, 4) \ge \begin{cases}
 {4v -21\over 3} &\mbox{ if $v \equiv 0 \bmod 3$}\\
 {4v -13\over 3} &\mbox{ if $v \equiv 1 \bmod 3$}\\
 {4v -17\over 3} &\mbox{ if $v \equiv 2 \bmod 3$}.
\end{cases}
\]
\end{Theorem}
\begin{proof}
The upper bound follows from Theorem \ref{t.bound}. The lower bounds follow from 
Theorems \ref{t.con}, \ref{t2.con} and \ref{t1.con}, using the fact that $D(4,4) = D(5,4) =1$.
\end{proof}

\begin{Theorem}\label{t.beta5}
Suppose $v \ge 360$.

 Then  
 \[ \beta(5, v, 4) \leq \left\lfloor {5v-15\over 3} \right\rfloor.\]
   Also, \[ \beta(5, v, 4) \ge \begin{cases}
 {5v-27\over 3} &\mbox{ if $v \equiv 0 \bmod 3$}\\
 {5v-32\over 3} &\mbox{ if $v \equiv 1 \bmod 3$}\\
 {5v-22\over 3} &\mbox{ if $v \equiv 2 \bmod 3$}.
\end{cases}
\]
\end{Theorem}
\begin{proof}
The upper bound follows from Theorem \ref{t.bound}. The lower bounds follow from 
Theorems \ref{t.con}, \ref{t2.con} and \ref{t1.con}, using the fact that $D(5,4) = D(6,4) = 1$.
\end{proof}

 \medskip
 
 We note that, using the bounds from Theorems \ref{t.bound1} and \ref{bound-improved}, one can also obtain
 results for smaller values of $v$.

% \medskip

For a $(v,4)$-packing, the largest possible value of $\rho$ is $\lfloor {v\over 4}\rfloor$.
 From the existence of $(v,4,1)$-BIBDs, we can determine some values of $\beta\left( \left\lfloor{v\over 4}\right\rfloor, v, 4\right)$.
 The necessary and sufficient conditions for the existence of a $(v,4,1)$-BIBD is $v \equiv 1, 4\bmod{12}$ (see \cite{BIBD}). Further, for 
 $v \equiv 4 \bmod{12}$, there exists a resolvable $(v,4,1)$-BIBD. 
 Now we consider the maximum partial parallel classes in $(v,4,1)$-BIBDs, for $v \equiv 1\bmod{12}$.

 A {\em $k$-group divisible design} (or \emph{$k$-GDD}) of type $h^n$ is a triple $(X, \mathcal G, \mathcal B)$, where $X$ is a  set
 of $hn$ \emph{points}, $\mathcal G$ is a partition of $X$ into $h$ \emph{groups} of size $n$ and $\mathcal B$ is a set of \emph{blocks} of 
 size $k$, such that every pair of distinct points of $X$ occurs in exactly one block or one group, but  not both.  
 A $k$-GDD of type $h^k$ is the same as a $TD(k,h)$. When the blocks
 in $\mathcal B$ can be partitioned into parallel classes, we say that the GDD is \emph{resolvable} GDD and denote it as a $k$-RGDD.
 
 From \cite[Theorem 2.19]{GDD}, we have the following result.
 
 \begin{Lemma}\label{l.GDD}
 If $n \equiv 0 \bmod 4$ and $n> 4$, then there exists a $4$-RGDD of type $3^n$.
 \end{Lemma}  
 
 By adding a new point to each of the groups of a $4$-RGDD of type $3^n$,
 we obtain a $(3n+1, 4,1)$-BIBD that has a PPC of size ${3n/4}$ (in fact, it has many PPCs of this size). 
 
 When $v  \equiv 1,  4\bmod {12}$, a $(v,4,1)$-BIBD is a maximum packing, so 
 $D(v, 4) = {v (v - 1)/12}$. So we have proven
 the following result.

\begin{Theorem}\label{t.betal}
Suppose  $v  \equiv 1,  4 \bmod {12}$ and $v \neq 13$.
 Then  
$$
 \beta\left( \left\lfloor{v\over 4}\right\rfloor, v, 4\right)  =  {v(v - 1) \over 12}. 
$$
\end{Theorem}

The exceptional case in Theorem \ref{t.betal} can  be handled easily.

\begin{Theorem}
$
 \beta\left( 3, 13, 4 \right)  = 7. 
$
\end{Theorem}

\begin{proof}
The following seven blocks are a packing on $13$ points with a maximum PPC of size $3$:
\[
\begin{array}{llll}
\{ 1,2,3,4\}, & \{5,6,7,8\}, & \{9,10,11,12\},\\
\{ 1,5,9,13\}, & \{2,6,10,13\}, & \{3,7,11,13\}, & \{4,8,12,13\}.
\end{array}
\]
Also, it is clear that there does not exist a packing having eight blocks and a maximum PPC of size $3$.
\end{proof}

 \section{Some results for $k > 4$}

Many of the methods used in previous sections can be generalized to $k > 4$.  First we consider constructions.
 
\begin{Theorem}\label{t.generalP}
Suppose there exist  $k - 2$ MOLS of order $n$ with a transversal of size $\rho \leq n$. Then there is 
a $((k-1)n+\rho, k)$-packing with $\rho n +  D(\rho ,k)$ blocks, in which   
  the size of largest PPC has  size $\rho$.
\end{Theorem}
\begin{proof}
The proof is very similar to the proof of Theorem \ref{l.RTD}.
We start with a TD$(k,n)$ having $\rho$ disjoint blocks. These blocks will be the PPC of size 
$\rho$. Let $Y$ denote the $\rho$ points in the last group 
that occur in a block of the PPC. Delete the blocks that do not contain a point in $Y$. The remaining $\rho n$ blocks form a packing on $(k-1)n +\rho$ points. We can also adjoin the blocks of a packing on $Y$. The resulting packing does not contain a PPC of size $\rho + 1$ because every block contains at least one point from $Y$.
\end{proof}

\begin{Corollary}\label{t.generalK}
Suppose there are $k - 1$ MOLS of order $n$. Then, 
for $1\leq \rho \le n$, there is a $((k- 1)n + \rho, k)$-packing
with $n\rho + D(\rho,k)$ blocks, in which the largest PPC has size $\rho$. 
\end{Corollary}

\begin{proof}
If there are $k - 1$ MOLS of order $n$, then any  $k - 2$ of these MOLS have a transversal of size $n$ and hence they have a transversal of size $\rho$ for any positive integer $\rho \leq n$. Apply Theorem \ref{t.generalP}.
\end{proof}

Next we  generalize Theorem \ref{t.bound1} in a straightforward manner to obtain an upper bound for $\beta(\rho, v, k)$. 
For a $(v, k)$-packing in which $\Par$ is a largest PPC of size $\rho$, let $P$ be the
 points in $\Par$ and let $T$ be the remaining points in the packing. 
 Consider a block $B = \{ a_1, a_2,\dots , a_k\} \in \Par$. For $1 \leq i \leq k$, let $T_{a_i}$ denote the set of blocks that contain $a_i$ and $k -1$ points in $T$. Denote
 $t_{a_i}= |T_{a_i}|$ for $1 \leq i \leq k$. Observe that \[t_{a_i} \leq \left\lfloor{v - k\rho \over k-1} \right\rfloor\] for all $i$.
 
 Similar to Lemma \ref{l-t_w}, we have
 
\begin{Lemma}
\label{l-t_w-k}
Suppose $B = \{ a_1, a_2,\dots , a_k\} \in \Par$ and suppose $t_{a_1} \geq \max \{  t_{a_2},\dots,t_{a_k}\}$. Then  one of  the following two conditions holds:
\begin{enumerate}
\item $t_{a_1} \leq k-1$, or
\item $t_{a_1} \geq k$ and  $t_{a_2} = \cdots = t_{a_k} = 0$.
\end{enumerate}
\end{Lemma}

\begin{Theorem}\label{t.boundK}
\[ \beta (\pc,v,k) \leq 
\rho \left( {k^2(\rho - 1)\over 2} + 1 + \max \left\{ k(k - 1), \left\lfloor{v - k\rho \over k-1} \right\rfloor \right\}\right).
\]
\end{Theorem}

\begin{proof}
Since $\Par$ is maximum, there is no block contained in $T$. 
From Lemma \ref{l-t_w-k}, 
there are at most $\pc \times \max \left\{ k(k-1), \left\lfloor \frac{v-k\pc}{k-1} \right\rfloor \right\}$
blocks having one point in $P$ and $k-1$ points in $T$. 

The number of pairs of points in $P$ that are not contained in a block of $\Par$ is 
\[ {k\rho \choose 2} - {k \choose 2}\rho = {k^2\rho (\rho -1)\over 2}.\]
Therefore there are at most $8\pc (\pc -1)$ 
blocks not in $\Par$ that contain at least two points in $P$. 
Finally, there are $\pc$ blocks in $\Par$.

In total,  there are at most 
\[
\rho \left( {k^2(\rho - 1)\over 2} + 1 + \max \left\{ k(k - 1), \left\lfloor{v - k\rho \over k-1} \right\rfloor \right\}\right)
\]
blocks in the packing.
\end{proof}

We now consider $\rho=1,2$ for general $k$. 

  \begin{Theorem}
 Suppose $v\equiv 1 \bmod (k-1)$ and $v \ge k(k-1)^2 + k$. Then   
 \[
 \beta(1, v, k) = {v - 1\over k - 1}.
 \]
 \end{Theorem}
 \begin{proof}
 Theorem \ref{t.boundK} yields the bound 
\[\beta ( 1, v, k) \le \left\lfloor {v - 1\over k - 1} \right\rfloor\] when $v \ge k(k-1)^2 + k$.
 When $v\equiv 1 \bmod (k-1)$, we can construct the desired packing by taking ${v - 1\over k - 1}$ blocks that contain a given point but are otherwise pairwise disjoint.
 \end{proof}

 For $\rho = 2$, 
 %Theorem \ref{t.generalK} gives a $(v, 4)$-packing with $2(v-2)\over k-1$ blocks. 
 the upper bound from Theorem \ref{t.boundK} is
 \begin{eqnarray*} \beta(2, v, k) &\leq& \left\lfloor {2(v -2k) \over k-1} \right\rfloor + k^2 + 2 \\&= &
 \left\lfloor {2v-4\over k-1}\right\rfloor  + k^2 - 2
\end{eqnarray*}
when $v \geq k(k-1)^2 +2k$.

On the other hand, we can construct a packing with ${2v-4\over k-1}$ blocks having a maximum PPC of size 2 whenever $v \equiv 2 \bmod (k-1)$ and $v \geq k(k-1) + 2$. Let $v = t(k-1)+2$ where $t \geq k$. We construct a packing on the points $(\zed_t \times \{1, \dots , k-1\}) \cup \{ \infty_1, \infty_2\}$. The packing has the following $2t$ blocks:
\[ \{ (0,1), (0,2), (0,3), \dots, (0,k-1) , \infty_1\} \bmod (t,-)\] \text{and} \[\{ (0,1), (1,2), (2,3) , \dots, (k-2,k-1) , \infty_2\} \bmod (t,-).\]
Since $t \geq k$, this packing contains two disjoint blocks:
\[(0,1), (0,2), (0,3), \dots, (0,k-1) , \infty_1\} \] \text{and} \[\{ (1,1), (2,2), (3,3) , \dots, (k-1,k-1) , \infty_2\}.\]
It is clear that the packing does not contain three disjoint blocks because every block contains
$\infty_1$ or $ \infty_2$.

 Thus we have proven the following.
 
  \begin{Theorem}
 Suppose $v \equiv 2 \bmod (k-1)$ and $v \geq k(k-1)^2 +2k$. Then 
 \[
 {2v-4\over k - 1} \le \beta(2, v, k) \le {2v-4\over k - 1} + k^2 - 2.
 \]
 \end{Theorem}
 
 Note that we proved a stronger result when $k=4$, for sufficiently large $v$, in Theorem \ref{t.beta2}.
 
 \section{Summary}
 
 In this paper, we studied $(v, 4)$-packings with maximum parallel classes of a pre-specified size and thus we
 extended the results of \cite{stinson} which studied this problem for $(v, 3)$-packings.
 
 We presented two constructions for $(v, 4)$-packings.
 However, our method using MOLS with disjoint transversals 
provides the greatest flexibility and we also used it to construct $(v, 4)$-packings with $k > 4$. 
 
% The mutual orthogonal
% idempotent  Latin squares might be useful for that construction. For example, we do not know if there are 3 mutual
% orthogonal Latin squares of order 10, but we have 2 mutual orthogonal idempotent 
% Latin squares of order 10. Therefore we can construct $(\rho, 30 + \rho, 4)$-packings based on
 %them. The constructions in Section \ref{s.con} give a lower bound $\beta_L( \rho, v, 4)$ 
% for $\beta( \rho, v, 4)$.
 
 Using counting arguments, we gave three upper bounds for $\beta(\rho, v, 4)$. While each successive bound improves the previous one, the latter bounds only hold for larger values of $v$. 
 
 Using the third upper bound,  we have 
 \[\beta(\rho, v, 4)  \leq {\rho v\over 3} - {\rho (\rho+1)\over 6} 
 = {\rho(v-\rho)\over 3} + {\rho (\rho - 1)\over 6}\] for sufficiently large values of $v$.  Our constructions give the lower bound 
 \[\beta (\rho, v, 4) \geq {\rho(v -\rho)\over 3} + D(v,4) \approx {\rho( v- \rho)\over 3} +{\rho (\rho-1)\over 12}.\]
 So our upper and lower bounds are very close, especially when $\rho $ is small. Also the difference between the upper and lower bounds is a constant (for a fixed value of $\rho$). 
 However, for small values of $v$ (or large values of $\rho$), the lower and/or upper bounds could potentially be improved.
 
 \section{Acknowledgement}
 
 We would like to thank Charlie Colbourn for informing us of reference \cite{GDD}, which we used in the 
 proof of Theorem \ref{t.betal}.

\end{document}